\documentclass{svjour3}                     
\smartqed  
\usepackage[CJKbookmarks=true,
            bookmarksnumbered=true,
            bookmarksopen=true,
            colorlinks=true,
            citecolor=blue,
            linkcolor=blue,
            anchorcolor=green,
            urlcolor=blue
            ]{hyperref}
\usepackage{hyperref}
\usepackage{color}

\usepackage{threeparttable}
\usepackage{fmtcount}  
\usepackage{amssymb}
\usepackage{amsmath}
\usepackage{bm}
\usepackage{booktabs}
\usepackage{graphicx}
\usepackage{subfigure}
\usepackage[latin1]{inputenc}
\usepackage[shortlabels]{enumitem}
\usepackage[numbers]{natbib}

\newcommand{\ud}{\mathrm{d}}
\journalname{Foundations of Computational Mathematics}
\begin{document}

\title{Minima distribution for global optimization}

\titlerunning{Minima distribution}        

\author{Xiaopeng Luo\thanks{The current version of this paper only discusses the theory of minima distribution; the relevant algorithm in the previous versions was separated and will be further expanded into a new upcoming paper.}}



\institute{X. Luo \at
              Department of Chemistry, Princeton University, Princeton, NJ 08544, USA \\
              \email{luo.permanent@gmail.com; xiaopeng@princeton.edu}
}

\date{}

\maketitle

\begin{abstract}
  This paper establishes a strict mathematical relationship between an arbitrary continuous function on a compact set and its global minima, like the well-known first order optimality condition for convex and differentiable functions. By introducing a class of nascent minima distribution functions that is only related to the target function and the given compact set, we construct a sequence that monotonically converges to the global minima on that given compact set. Then, we further consider some various sequences of sets where each sequence monotonically shrinks from the original compact set to the set of all global minimizers, and the shrink rate can be determined for continuously differentiable functions. Finally, we provide a different way of constructing the nascent minima distribution functions.


\end{abstract}

\section{Introduction}
\label{MD:s1}

Given a possibly highly nonlinear and non-convex continuous function $f:\Omega\subset\mathbb{R}^n\to\mathbb{R}$ with the global minima $f^*$ and the set of all global minimizers $X^*$ in $\Omega$, we consider the optimality condition for the constrained optimization problem
\begin{equation}\label{MD:eq:COP}
  \min_{x\in\Omega}f(x),
\end{equation}
that is, what the clear relationship between $f$ and $f^*$ (or $X^*\subseteq\Omega$)? Here, $\Omega$ is a (not necessarily convex) compact set defined by inequalities $g_i(x)\geqslant0, i=1,\cdots,r$.

Generally, finding an arbitrary local minima is relatively straightforward by using classical local algorithms; however, finding the global minima is much more difficult. According to the degree of utilization of the global prior information coming from previous function evaluations, all the existing global algorithms could be roughly divided into three categories. First, there are those with almost no use of prior information. A typical representative is the so-called multistart algorithm based on the idea of performing parallel local searches starting from multiple initial points \cite{BoenderC1982M_Multistart,RinnooyKan1987M_Multistart1, RinnooyKan1987M_Multistart2,ByrdR1990M_Multistart}. It is usually effective if the number of local minimizers of a target function is not large; however, one cannot see any overall landscape in the multistart algorithm since there is no information exchange between those parallel local searches. Actually, this is a common feature of the traditional random search methods \cite{SchumerM1968M_RSstep,SchrackG1976M_RSstep,SheelaB1979M_RSstep, MasriS1980M_ASR} that appeared in the 1950s \cite{AndersonR1953M_RS, BrooksS1958M_RS,RastriginL1960M_RS,RastriginL1963A_RS,MutseniyeksV1964A_RS}.

Second, there are those with only partial use of prior information, including many heuristic algorithms. Central to these methods is a strategy that generates variations of a set of candidates (often called a population), and the information exchange among population members happens to be the focus of attention. The best known of these are genetic algorithms (GA) \cite{GoldbergD1989_GA,MitchellM1996_GA}, evolution strategies (ES) \cite{RechenbergI1973_ES,SchwefelH1995_ES}, differential evolution (DE) \cite{StornR1997M_DE,PriceK2005_DE,DasS2011S_DE, DasS2016S_DE} and so forth \cite{YangXS2014B_NIOA}. DE usually performs well for continuous optimization problems \cite{PriceK1996M_DE96,StornR1997M_DE,DasS2011S_DE,DasS2016S_DE} although does not guarantee that an optimal solution is ever found. The evolution strategy of DE takes the difference of two randomly chosen candidates to perturb an existing candidate and accepts a new candidate under a greedy or annealing criterion \cite{KirkpatrickS1983M_SimulatedAnnealing}.

Finally, there are those with full use of prior information, such as Bayesian optimization (BO) \cite{MockusJ1974M_BO,MockusJ1978A_BO,JonesD1998M_BO, ShahriariB2016_BO}. The BO method is to treat a target as a random function with a prior distribution and applies Bayesian inference to update the prior according to the previous function observations. This updated prior is used to construct an acquisition function to determine the next candidate. The acquisition function, which trade-offs exploration and exploitation, can be of different types, such as probability of improvement (PI) \cite{KushnerH1964_BO_PI}, expected improvement (EI) \cite{MockusJ1978A_BO,BullA2011_BO_EI_Conv} or lower confidence bound (LCB) \cite{CoxD1997M_BO_UCB}. BO is a sequential model-based approach and the model is often obtained using a Gaussian process (GP) \cite{JonesD1998M_BO, FloudasCA2008R_GlobalOptimization,RiosLM2013R_GlobalOptimization} which provides a normally distributed estimation of the target function \cite{RasmussenC2006_GP}.

Although most global methods try to extract as much knowledge as possible from prior information, the connection between the entire landscape of a target function and its global minima is not yet sufficiently clear and precise. Specifically, there is currently the lack of such an essential mathematical relationship; in contrast, there is a well-known relationship between a differentiable convex function and its minima, which is established by the gradient. Suppose $f$ is differentiable on a convex set $D$, then $f$ is convex if and only if
\begin{equation*}
  f(y)\geqslant f(x)+\nabla f(x)(y-x),~~\forall x,y\in D;
\end{equation*}
see for example Ref. \cite{BoydS2004B_ConvexOptimization}. And this shows that $\nabla f(x^*)=0$ implies that $f(y)\geqslant f(x^*)$ holds for all $y\in D$, i.e., $x^*$ is a global minimizer of $f$ on $D$. The role of this essential correlation is self-evident in convex optimization.

This paper aims to establish a similar mathematical relationship between any continuous function $f$ on a compact set $\Omega\subset\mathbb{R}^n$ and its global minima $f^*$. As the main contributions of this paper, if $m^{(k)}(x)=\frac{e^{-kf(x)}}{\int_\Omega e^{-kf(t)}\ud t}$ for $k\in\mathbb{R}$, then we have
\begin{itemize}
\renewcommand{\labelitemi}{$\bullet$}
  \item It holds that
    \begin{equation}
      \lim_{k\to\infty}\int_\Omega f(x)m^{(k)}(x)\ud x=f^*,
    \end{equation}
    and
    \begin{equation}
      \lim_{k\to\infty}m^{(k)}(x)
       =\left\{\begin{array}{cl}
       \frac{1}{\mu(X^*)}, & x\in X^*, \\
       0, & x\in\Omega-X^*,
      \end{array}\right.
    \end{equation}
    where $\mu(X^*)$ is the $n$-dimensional Lebesgue measure of $X^*$ and $X^*$ is the set of all global minimizers on $\Omega$.
  \item If $f$ is not a constant function on $\Omega$, the monotonic relationship
    \begin{equation}
      \int_\Omega f(x)m^{(k)}(x)\ud x>
      \int_\Omega f(x)m^{(k+\Delta k)}(x)\ud x>f^*
    \end{equation}
    holds for all $k\in\mathbb{R}$ and $\Delta k>0$, which implies a series of monotonous containment relationships, for examples, it holds that
    \begin{equation}
     \Omega\!\supset\!D_f^{(k)}\!\supset\!D_f^{(k+\Delta k)}\!\supset\!X^*,
     ~\textrm{where}~D_f^{(k)}\!=\!\left\{x\!\in\!\Omega:f(x)\!\leqslant\!
     \int_\Omega\!f(t)m^{(k)}(t)\ud t\right\},
    \end{equation}
    and for all $k\geqslant0$ and $\Delta k>0$, it holds that
    \begin{equation}
     \Omega=D_0^{(0)}\!\supset\!D_0^{(k)}\!
     \supset\!D_0^{(k+\Delta k)}\!\supset\!X^*,
     ~\textrm{where}~D_0^{(k)}\!
     =\!\left\{x\!\in\!m^{(k)}(t)\geqslant1/\mu(\Omega)\right\}
    \end{equation}
  \item If $x^*$ is the unique global minimizer of $f$ in $\Omega$, then
    \begin{equation}
      \lim_{k\to\infty}\int_\Omega x\cdot m^{(k)}(x)\ud x=x^*;
    \end{equation}
    further, if $f(x-x^*)=\psi(\|x-x^*\|_2)$ is radial and $\psi$ is a non-decreasing function on $\mathbb{R}_+$, then for all $k\in\mathbb{R}$ and $\Delta k>0$, it holds that
    \begin{equation}
      \left\|\int_\Omega(x-x^*)m^{(k)}(x)\ud x\right\|_2\geqslant
      \left\|\int_\Omega(x-x^*)m^{(k+\Delta k)}(x)\ud x\right\|_2.
    \end{equation}
  \item Suppose $f\in C(\Omega)$ and $\Gamma_0^{(k)}=\{x\in\Omega: m^{(k)}(x)=1/\mu(\Omega)\}$. If $x\in\Gamma_0^{(k)}$ and $x$ moves to $x+\Delta x\in\Gamma_0^{(k+\Delta k)}$ when $k$ continuously increases to $k+\Delta k$, then
    \begin{equation}
      \lim_{\Delta k\to0}\frac{f(x)-f(x+\Delta x)}{\Delta k}
      =\frac{1}{k}\left(f(x)-\int_\Omega f(t)m^{(k)}(t)\ud t\right),
    \end{equation}
    and
    \begin{equation}
      \lim_{\Delta k\to0}\frac{\|\Delta x\|}{\Delta k}
      =\frac{1}{k\|\nabla f(x)\|}
      \left|f(x)-\int_\Omega f(t)m^{(k)}(t)\ud t\right|,
    \end{equation}
    where $\|\cdot\|$ is the Euclidean norm.
\end{itemize}

The remainder of the paper is organized as follows. In Sect. \ref{MD:s2}, the concept of minima distribution (MD) and relevant conclusions are fully built by introducing a class of nascent minima distribution functions. In Sect. \ref{MD:s3}, we consider some various sequences of sets where each sequence monotonically shrinks from the original compact set to the set of all global minimizers. In Sect. \ref{MD:s4}, we provide another different way of constructing the nascent minima distribution functions. And finally, we draw some conclusions in Sect. \ref{MD:s5}.

\section{Minima distribution}
\label{MD:s2}

To establish a mathematical relationship between $f$ and $f^*$, we hope to find an integrable distribution function $m_{f,\Omega}\in\mathcal{L}(\Omega)$ such that
\begin{equation*}
  f^*=\int_\Omega f(x)m_{f,\Omega}(x)\ud x
  ~~\textrm{and}~~m_{f,\Omega}(x)>0~\textrm{if}~x\in X^*,
  ~m_{f,\Omega}(x)=0~\textrm{if}~x\in\Omega-\!X^*.
\end{equation*}
Since the distribution $m_{f,\Omega}$ is closely linked to the minimization of function $f$ on $\Omega$, we call it a \emph{minima distribution} related to $f$ and $\Omega$. In the following, we will first introduce the nascent minima distribution function $m^{(k)}$ related to $f$ and $\Omega$ then define the minima distribution (MD) by a weak limit of $m^{(k)}$.

\subsection{Nascent minima distribution function}

Our motivation for introducing nascent minima distribution functions originated with a meaningful observation. Consider the function
\begin{equation}\label{MD:eq:motivation}
  \tau(x)=\frac{1}{f(x)-f^*+1},~x\in\Omega,
\end{equation}
then $\tau(x)=1$ for $x\in X^*$ while $0<\tau(x)<1$ for $x\in\Omega-X^*$; and then,
\begin{equation*}
  \lim_{k\to\infty}\tau^k(x)=
  \left\{\begin{array}{cc}
    1, & x\in X^*; \\
    0, & x\notin X^*.
  \end{array}\right.
\end{equation*}

\begin{figure}[!h]
\centering
\includegraphics[width=0.24\textwidth]{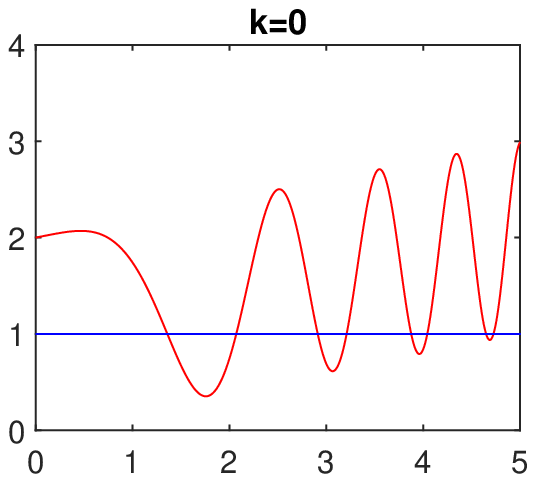}
\includegraphics[width=0.24\textwidth]{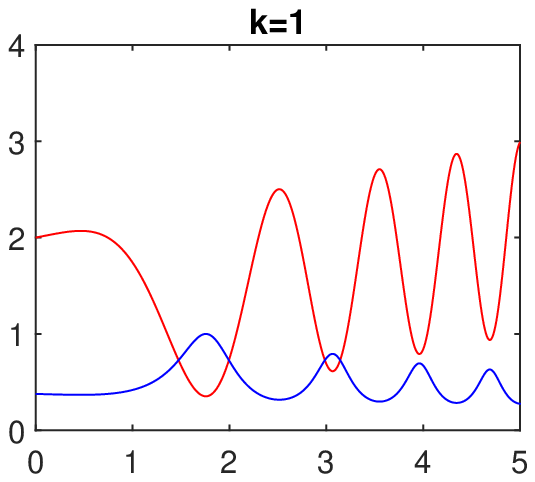}
\includegraphics[width=0.24\textwidth]{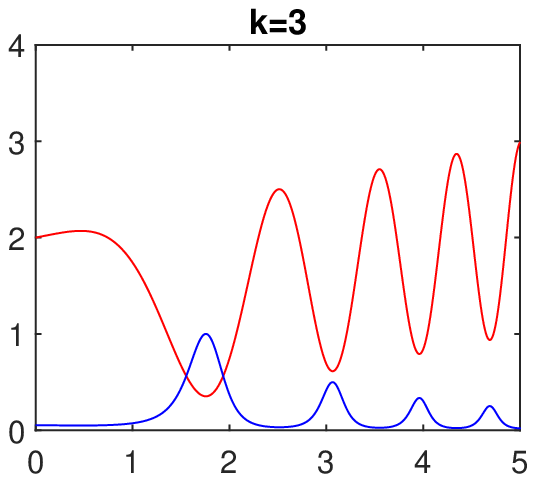}
\includegraphics[width=0.24\textwidth]{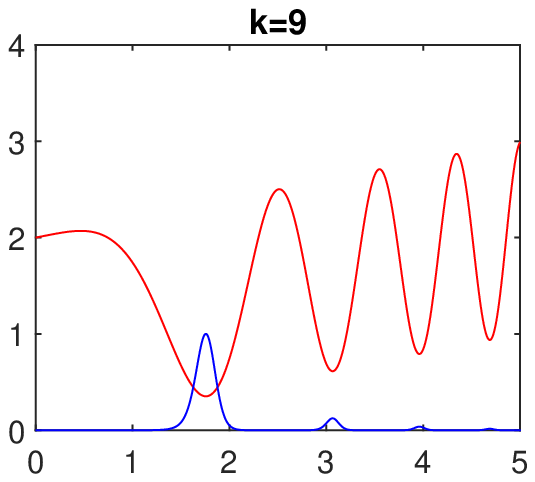}
\caption{One-dimensional illustration of the functions $\tau^k(x)=(f(x)-f^*+1)^{-k}$. The plots show $f(x)=\cos(x^2)+x/5+1,x\in[0,5]$ in red and the relevant $\tau^k$ with $k=0,1,3,9$ in blue.}
\label{MD:fig:motivation1}
\end{figure}

\begin{figure}[!h]
\centering
\subfigure[$f$]{\includegraphics[width=0.24\textwidth]{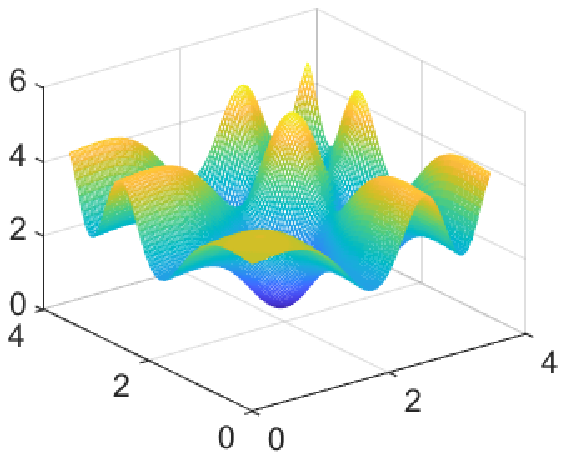}}
\subfigure[$\tau^1$]{\includegraphics[width=0.24\textwidth]{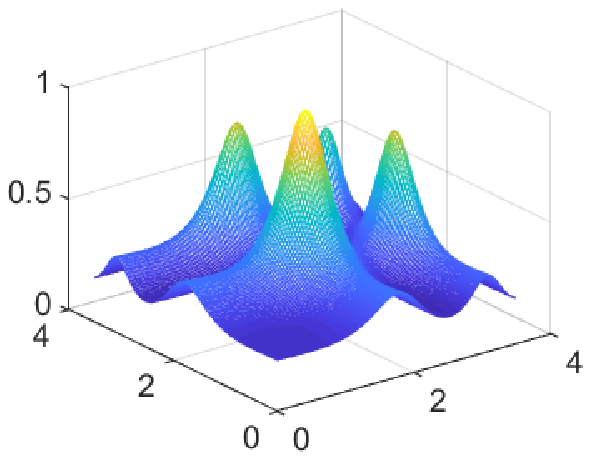}}
\subfigure[$\tau^3$]{\includegraphics[width=0.24\textwidth]{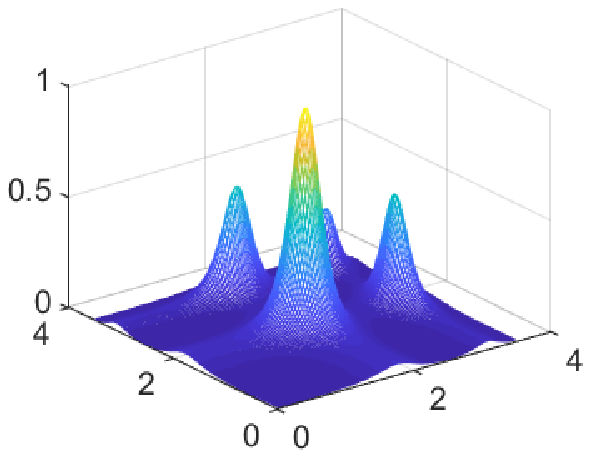}}
\subfigure[$\tau^9$]{\includegraphics[width=0.24\textwidth]{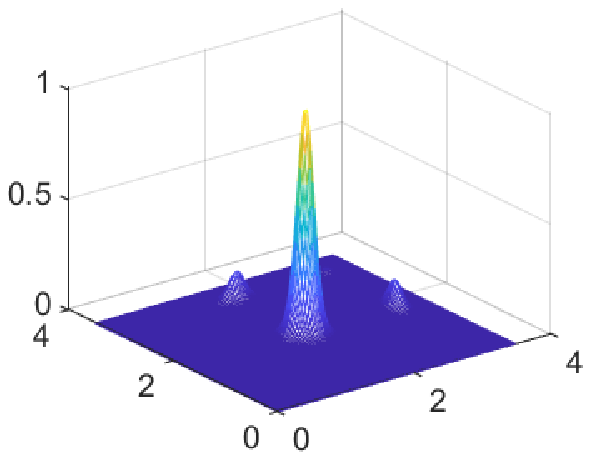}}
\caption{Two-dimensional illustration of the functions $\tau^k(x)=(f(x)-f^*+1)^{-k}$. The plot (a) shows the target function $f(x)=\cos(x_1^2)+\cos(x_2^2)+x_1/5+x_2/5+2,x=(x_1,x_2)\in[0,3.5]^2$, and the plots (b)-(d) show the relevant $\tau^k$ with $k=1,3,9$.}
\label{MD:fig:motivation2}

\end{figure}
The evolution of $\tau^k$ changing with $k$ can be clearly illustrated in Figs. \ref{MD:fig:motivation1} and \ref{MD:fig:motivation2} for $1$-dimensional and $2$-dimensional cases, respectively. According to the property described above, it is reasonable to expect that
\begin{equation*}
  f^*=\lim_{k\to\infty}
  \frac{\int_\Omega f(x)\tau^k(x)\ud x}
  {\int_\Omega\tau^k(t)\ud t}.
\end{equation*}
It is worth noting that the identity above depends only on the monotonicity and nonnegativity of the function $\rho(y)=\frac{1}{y-f^*+1}$, that is, $\tau(x)=\rho(f(x))$. The reason for including $f^*$ in the definition of $\tau$ is only to meet the nonnegativity requirement. So we can introduce the following concept:
\begin{definition}\label{MD:def:nmd}
Suppose $\Omega\subset\mathbb{R}^n$ is a compact set, $f:\Omega\subset\mathbb{R}^n\to\mathbb{R}$ is a continuous real function on $\Omega$, i.e., $f\in C(\Omega)$, and $\rho:\mathbb{R}\to\mathbb{R}$ is monotonically decreasing with $\rho(f(x))>0$ for every $x\in\Omega$. For any $k\in\mathbb{R}$, we define a nascent minima distribution function by
\begin{equation}\label{MD:eq:nmd}
  m^{(k)}(x)=m_{f,\Omega}^{(k)}(x)=
  \frac{\tau^k(x)}{\int_\Omega\tau^k(t)\ud t},
  ~~\textrm{where}~~\tau(x)=\rho(f(x)).
\end{equation}
And a typical choice of $\tau$ is the exponential-type, i.e., $\tau(x)=e^{-f(x)}$.
\end{definition}
\begin{remark}
The exponential-type definition of $\tau$ does not depend on $f^*$ because of the nonnegativity of the exponential function itself. In addition, the rational-type definition \eqref{MD:eq:motivation} can be further extended as $\tau(x)=\frac{1}{f(x)-f^*+p}$ for any $p>0$, and the arbitrariness of $p$ could partially weaken the dependence of $\tau$ on the unknown $f^*$.
\end{remark}

\begin{figure}[!h]
\centering
\includegraphics[width=0.24\textwidth]{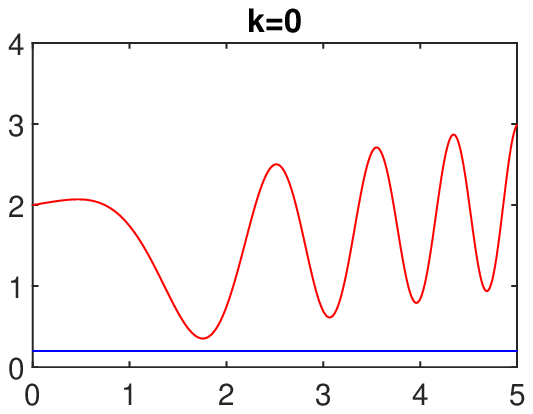}
\includegraphics[width=0.24\textwidth]{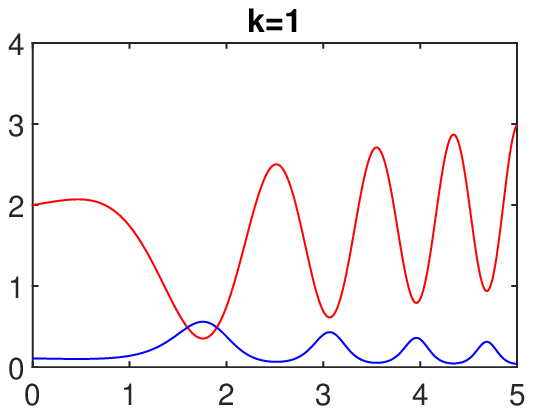}
\includegraphics[width=0.24\textwidth]{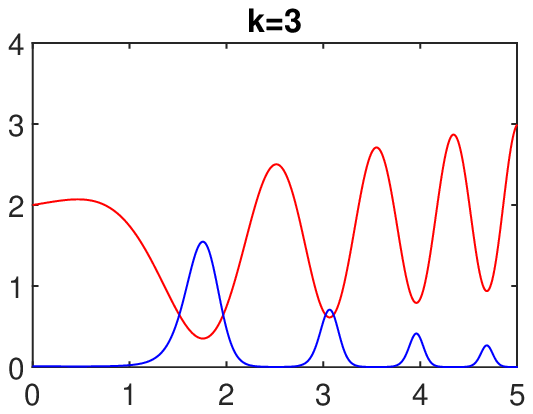}
\includegraphics[width=0.24\textwidth]{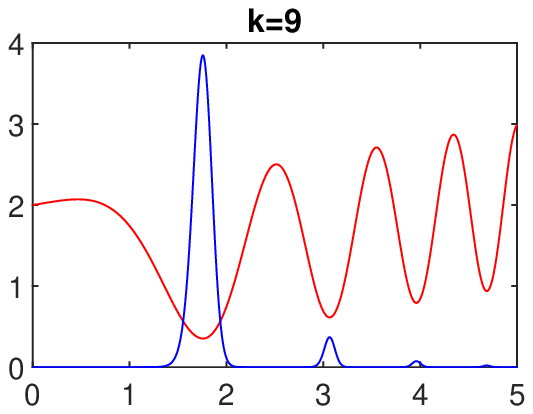}
\caption{One-dimensional illustration of the nascent MD functions. The plots show the target function $f=\cos(x^2)+x/5+1,x\in[0,5]$ in red and the relevant exponential-type nascent MD functions $m^{(k)}$ with $k=0,1,3,9$ in blue.}
\label{MD:fig:nmd1}
\end{figure}

\begin{figure}[!h]
\centering
\subfigure[$f$]{\includegraphics[width=0.24\textwidth]{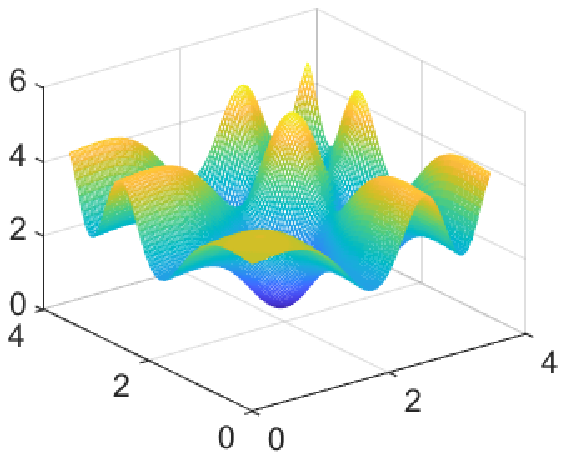}}
\subfigure[$m^{(1)}$]{\includegraphics[width=0.24\textwidth]{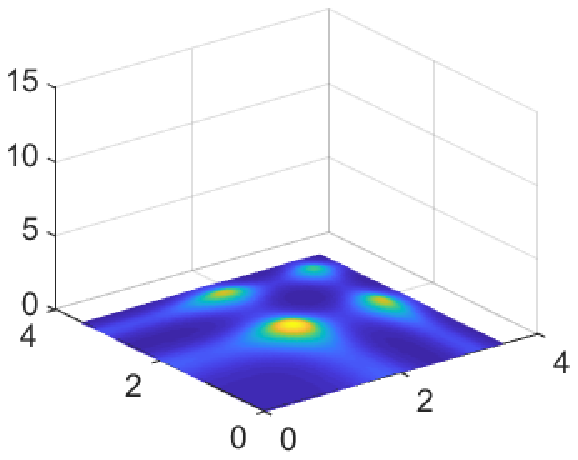}}
\subfigure[$m^{(3)}$]{\includegraphics[width=0.24\textwidth]{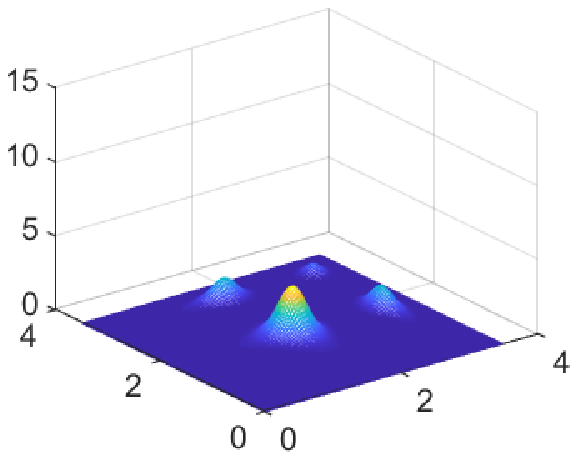}}
\subfigure[$m^{(9)}$]{\includegraphics[width=0.24\textwidth]{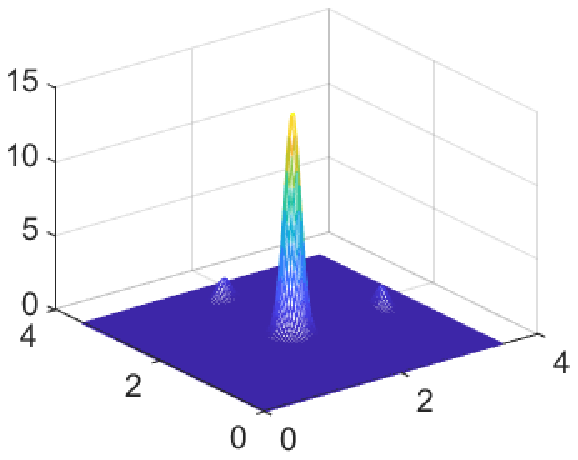}}
\caption{Two-dimensional illustration of the nascent MD functions. The plot (a) shows the target function $f(x)=\cos(x_1^2)+\cos(x_2^2)+x_1/5+x_2/5+2, x=(x_1,x_2)\in[0,3.5]^2$, and the plots (b)-(d) show the relevant exponential-type nascent MD functions $m^{(k)}$ with $k=1,3,9$.}
\label{MD:fig:nmd2}
\end{figure}

The nascent MD functions of varying parameters are illustrated in Figs. \ref{MD:fig:nmd1} and \ref{MD:fig:nmd2} for $1$-dimensional and $2$-dimensional cases, respectively. In each example, $X^*=x^*$ is a single point set and $m^{(k)}$ gradually evolves from a uniform distribution on $\Omega$ to the Dirac delta function $\delta(x-x^*)$; however, the construction of $m^{(k)}$ could be totally not related to the point $x^*$. This implies a potential relationship between $f$ and $X^*$, which will be strictly established later. And $X^*\subset\Omega$ could be not only any single point, finite, or countable subset but also any measurable subset.

For any $h\in C(\Omega)$ and $\nu\in\mathbb{R}$, define the expectation
\begin{equation}\label{MD:eq:E}
  \mathbb{E}^{(k)}_t\big(h(x+t)\big)^\nu
  =\int_\Omega h^\nu(x+t)m^{(k)}(t)\ud t
\end{equation}
and denote $\mathbb{E}^{(k)}(h^\nu)=\mathbb{E}^{(k)}_t\big(h(t)\big)^\nu$ to unclutter the notation.

First, we consider some relevant properties of the nascent MD functions.
\begin{theorem}\label{MD:thm:nmd}
The nascent minima distribution function defined in \eqref{MD:eq:nmd} satisfies:
\begin{enumerate}[(i)]
  \item For $k>0$, the maxima of $m^{(k)}$ is $m^{(k)}(x^*)$ and the set of all maximizers is $X^*$.
  \item For any $k\in\mathbb{R}$, $m^{(k)}$ is a probability density function (PDF) on $\Omega$; especially, $m^{(0)}(x)=1/\mu(\Omega)$ is the PDF of the continuous uniform distribution on $\Omega$, where $\mu(S)=\int_S\ud t$ is the $n$-dimensional Lebesgue measure of $S$.
  \item If $f\in C^1(\Omega)$, then $\nabla m^{(k)}(x)=km^{(k)}(x)\nabla \tau(x)/\tau(x)$.
  \item Suppose $f\in C^2(\Omega)$. If $x^*_1,x^*_2\in X^*$ and $\nabla^2f(x^*_1)\succ\nabla^2f(x^*_2)\succ0$, then
      \begin{equation*}
        \nabla^2m^{(k)}(x_1^*)\prec\nabla^2m^{(k)}(x_2^*)\prec0,
      \end{equation*}
      where $A\succ0$ is a generalized inequality meaning $A$ is a strictly positive definite matrix, $\nabla^2f$ and $\nabla^2m^{(k)}$ are the Hessian functions of $f$ and $m^{(k)}$, respectively.
  \item For every $k\in\mathbb{R}$, it holds that
      \begin{equation*}
        \frac{\ud}{\ud k}m^{(k)}(x)=m^{(k)}(x)
        \Big(\log(\tau(x))-\mathbb{E}^{(k)}(\log(\tau))\Big).
      \end{equation*}
    \item If $X^*$ has zero $n$-dimensional Lebesgue measure, i.e., $\mu(X^*)=0$, then
      \begin{equation*}
        \lim_{k\to\infty}m^{(k)}(x)=
        \left\{\begin{array}{cc}
        \infty, & x\in X^*; \\
        0, & x\notin X^*.
        \end{array}\right.
      \end{equation*}
  \item If $X^*$ has nonzero $n$-dimensional Lebesgue measure, i.e., $\mu(X^*)>0$, then
      \begin{equation*}
        \lim_{k\to\infty}m^{(k)}(x)=
        \left\{\begin{array}{cc}
        \frac{1}{\mu(X^*)}, & x\in X^*; \\
        0, & x\notin X^*.
        \end{array}\right.
      \end{equation*}
\end{enumerate}
\end{theorem}
\begin{proof}
Clearly, (i) and (ii) follow from the monotonicity and nonnegativity of $\rho$.

If $f\in C^1(\Omega)$, then (iii) follows from
\begin{equation*}
  \nabla m^{(k)}(x)=\frac{\nabla\tau^k(x)}{\int_\Omega\tau^k(t)\ud t}
  =\frac{k\tau^{k-1}(x)\nabla\tau(x)}{\int_\Omega\tau^k(t)\ud t}
  =\frac{km^{(k)}(x)\nabla\tau(x)}{\tau(x)},
\end{equation*}

If $f\in C^2(\Omega)$, for any $x^*\in X^*$, it holds from $\nabla\tau(x^*)=\rho'(f^*)\nabla f(x^*)=0$ that
\begin{equation*}
  \nabla^2m^{(k)}(x^*)
  =\rho'(f^*)\frac{k\tau^{k-1}(x^*)}
  {\int_\Omega\tau^k(t)\ud t}\nabla^2f(x^*),
\end{equation*}
then (iv) follows from the monotonicity of $\rho$, i.e., $\rho'(f^*)<0$.

For every $k\in\mathbb{R}$, it holds from $\frac{\ud}{\ud k}\tau^k(x)= \tau^k(x)\log(\tau(x))$ that
\begin{align*}
  \frac{\ud}{\ud k}m^{(k)}(x)=&\frac{\ud}{\ud k}
  \left(\frac{\tau^k(x)}{\int_\Omega\tau^k(t)\ud t}\right) \\
  =&\frac{\tau^k(x)\log(\tau(x))\int_\Omega\tau^k(t)\ud t-
  \tau^k(x)\int_\Omega\tau^k(t)\log(\tau(t))\ud t}
  {\left(\int_\Omega\tau^k(t)\ud t\right)^2} \\
  =&m^{(k)}(x)\left(\log(\tau(x))-\int_\Omega m^{(k)}(t)\log(\tau(t))\ud t\right) \\
  =&m^{(k)}(x)\Big(\log(\tau(x))-\mathbb{E}^{(k)}(\log(\tau))\Big),
\end{align*}
this proves (v).

For any $x'\notin X^*$, let $p=\tau(x')>0$, then there must exist an open set $\Omega_p$ that has nonzero $n$-dimensional Lebesgue measure such that $\tau(t)>p$ if $t\in\Omega_p$ and $\tau(t)\leqslant p$ if $t\notin\Omega_p$, and further,
\begin{align*}
  m^{(k)}(x')=\frac{p^k}{\int_{\Omega_p}\tau^k(t)\ud t
  +\int_{\Omega-\Omega_p}\tau^k(t)\ud t}
  \leqslant\frac{1}{\int_{\Omega_p}p^{-k}\tau^k(t)\ud t},
\end{align*}
since $p^{-1}\tau(t)>1$ for any $t\in\Omega_p$, the limit of $\int_{\Omega_p}p^{-k}\tau^k(t)\ud t$ tends to $\infty$ as $k\to\infty$; thus, it holds that
\begin{equation*}
  \lim_{k\to\infty}m^{(k)}(x')=0,~\forall x'\notin X^*.
\end{equation*}
Otherwise, for any $x''\in X^*$,
\begin{align*}
  m^{(k)}(x'')=\frac{1}{\int_{X^*}\ud t
  +\int_{\Omega-X^*}\tau^{-k}(x'')\tau^k(t)\ud t},
\end{align*}
since $\tau^{-1}(x'')\tau(t)<1$ for any $t\in\Omega-X^*$, the limit of $\int_{\Omega-X^*}\tau^{-k}(x'')\tau^k(t)\ud t$ tends to $0$ as $k\to\infty$; thus, for any $x''\in X^*$, it follows that
\begin{equation*}
 \lim_{k\to\infty}m^{(k)}(x'')=
 \left\{\begin{array}{cc}
   \infty, & \mu(X^*)=0, \\
   \frac{1}{\mu(X^*)}, & \mu(X^*)>0;
 \end{array}\right.
\end{equation*}
this proves (vi) and (vii), and the proof is complete.\qed
\end{proof}
\begin{remark}
According to the properties (iii) and (v), it is a very natural thing to choose the exponential-type $\tau$; and in this case,
\begin{equation*}
  \nabla m^{(k)}(x)=-km^{(k)}(x)\nabla f(x)~~\textrm{and}~~
  \frac{\ud}{\ud k}m^{(k)}(x)=m^{(k)}(x)\left(\mathbb{E}^{(k)}(f)-f(x)\right).
\end{equation*}
\end{remark}

\subsection{Monotonic convergence}
\label{MD:s2:mc}

An attractive property of $m^{(k)}$ is that $\int_\Omega f(x)m^{(k)}(x) \ud x$ monotonically converges to the global minima $f^*$ as $k\to\infty$ for all continuous functions on a compact set $\Omega$ without any other assumptions. This monotonic convergence strictly confirms the role of $m^{(k)}$ as a link between $f$ and $f^*$.

We first prove the convergence according to the continuity of $f$.
\begin{theorem}[convergence]\label{MD:thm:C}
If $f\in C(\Omega)$, then
\begin{equation*}
  \lim_{k\to\infty}\int_\Omega f(x)m^{(k)}(x)\ud x=f^*;
\end{equation*}
moreover, if $x^*$ is the unique global minimizer of $f$ in $\Omega$, then
\begin{equation*}
  \lim_{k\to\infty}\int_\Omega x\cdot m^{(k)}(x)\ud x=x^*.
\end{equation*}
\end{theorem}
\begin{proof}
If $f$ is a constant function on $\Omega$, for every $k$, we have
\begin{equation*}
  \int_\Omega f(x)m^{(k)}(x)\ud x=f^*.
\end{equation*}
Now suppose that $f$ is not a constant function on $\Omega$. Given $\epsilon>0$, it follows from the continuity of $f$ that there exists an open set $\Omega_\epsilon$ such that $X^*\subset\Omega_\epsilon$ and $f(x)-f^*<\epsilon$ holds for all $x\in\Omega_\epsilon$; further, according to (vi) and (vii) of Theorem \ref{MD:thm:nmd}, there exists a $K\in\mathbb{N}$ such that
\begin{equation*}
  \int_{\Omega-\Omega_\epsilon}m^{(k)}(x)\ud x<\epsilon
\end{equation*}
holds for every $k>K$, hence,
\begin{align*}
  \int_\Omega f(x)m^{(k)}(x)\ud x-f^*
  =&\int_\Omega\big(f(x)-f^*\big)m^{(k)}(x)\ud x \\
  =&\int_{\Omega_\epsilon}\!\!\big(f(x)\!-\!f^*\big)m^{(k)}(x)\ud x
  +\int_{\Omega-\Omega_\epsilon}\!\!\big(f(x)\!-\!f^*\big)m^{(k)}(x)\ud x \\
  <&\epsilon(1-\epsilon)+R_f\epsilon<(1+R_f)\epsilon,
\end{align*}
where $R_f=\max_{x\in\Omega}f(x)-f^*$ and $k>K$, which completes the proof of the first identity; and in a similar way one can establish the second one.\qed
\end{proof}

Now we prove the nonnegativity.
\begin{theorem}[nonnegativity]\label{MD:thm:NN}
For every $k\in\mathbb{R}$, we have
\begin{equation*}
  \int_\Omega f(x)m^{(k)}(x)\ud x\geqslant f^*,
\end{equation*}
and it becomes an equality if and only if $f$ is a constant function on $\Omega$.
\end{theorem}
\begin{proof}
According to (ii) of Theorem \ref{MD:thm:nmd}, $\int_\Omega m^{(k)}(x)\ud x=1$ for every $k\in\mathbb{R}$, so it follows that
\begin{equation*}
  \int_\Omega f(x)m^{(k)}(x)\ud x-f^*
  =\int_\Omega\big(f(x)-f^*\big)m^{(k)}(x)\ud x,
\end{equation*}
hence, 
\begin{equation*}
  \int_\Omega f(x)m^{(k)}(x)\ud x-f^*=0,
\end{equation*}
if and only if $f$ is a constant function on $\Omega$. 
If $f$ is not a constant function on $\Omega$, there exists a domain $D\subset\Omega$ of nonzero measure such that $f(x)-f^*>0$ on $D$; together with the nonnegativity of $m^{(k)}$, we have
\begin{equation*}
  \int_\Omega\big(f(x)-f^*\big)m^{(k)}(x)\ud x
  \geqslant\int_D\big(f(x)-f^*\big)m^{(k)}(x)\ud x>0,
\end{equation*}
and the proof is complete.\qed
\end{proof}

To prove the monotonicity, we need the following lemma.
\begin{lemma}[Gurland's inequality \cite{GurlandJ1967A_ExpectationInequality}]\label{MD:lem:Gurland}
Suppose $y$ is an arbitrary random variable defined on a subset of $\mathbb{R}$, if $g$ and $h$ are both non-increasing or non-decreasing, then $\mathbb{E}\big(g(y)\cdot h(y)\big)\geqslant \mathbb{E}(g(y))\mathbb{E}\big(h(y)\big)$; if $g$ is non-decreasing and $h$ non-increasing, or vice versa, then $\mathbb{E}\big(g(y)\cdot h(y)\big) \leqslant\mathbb{E}(g(y))\mathbb{E}\big(h(y)\big)$.
\end{lemma}
\begin{theorem}[monotonicity]\label{MD:thm:MC}
For $k\in\mathbb{R}$, the nascent minima distribution function defined in \eqref{MD:eq:nmd} satisfies
\begin{equation*}
  \frac{\ud}{\ud k}\int_\Omega f(x)m^{(k)}(x)\ud x\leqslant0;
\end{equation*}
especially, if $\tau$ is the exponential-type $e^{-f}$ and $f$ is not a constant function on $\Omega$, then
\begin{equation*}
  \frac{\ud}{\ud k}\int_\Omega f(x)m^{(k)}(x)\ud x=-\mathbb{V}ar^{(k)}(f)<0,
\end{equation*}
where $\mathbb{V}ar^{(k)}(f)= \int_\Omega \big(f(x)-\mathbb{E}^{(k)}(f)\big)^2 m^{(k)}(x)\ud x$.
\end{theorem}
\begin{proof}
According to (v) of Theorem \ref{MD:thm:nmd}, we have
\begin{align*}
  m^{(k+\Delta k)}(x)=&m^{(k)}(x)+
  \int_k^{k+\Delta k}\frac{\ud m^{(v)}(x)}{\ud v}\ud v\\
  =&m^{(k)}(x)+\int_k^{k+\Delta k}m^{(v)}(x)
  \left(\log(\tau(x))-\mathbb{E}^{(v)}(\log(\tau))\right)\ud v,
\end{align*}
then there exists a $\zeta\in(k,k+\Delta k)$ such that
\begin{align*}
  \frac{\mathbb{E}^{(k+\Delta k)}(f)-\mathbb{E}^{(k)}(f)}{\Delta k}
  =&\frac{1}{\Delta k}\int_\Omega f(x)
  \left(m^{(k+\Delta k)}(x)-m^{(k)}(x)\right)\ud x \\
  =&\frac{1}{\Delta k}\int_\Omega\int_k^{k+\Delta k}f(x)m^{(v)}(x)
  \left(\log(\tau(x))-\mathbb{E}^{(v)}\big(\log(\tau)\big)\right)\ud v\ud x \\
  =&\frac{1}{\Delta k}\int_k^{k+\Delta k}
  \left(\mathbb{E}^{(v)}\big(f\log(\tau)\big)
  -\mathbb{E}^{(v)}(f)\mathbb{E}^{(v)}\big(\log(\tau)\big)\right)\ud v \\
  =&\mathbb{E}^{(\zeta)}\big(f\log(\tau)\big)
  -\mathbb{E}^{(\zeta)}(f)\mathbb{E}^{(\zeta)}\big(\log(\tau)\big),
\end{align*}
hence, we have
\begin{align}\label{MD:eq:Ef}
  \frac{\ud \mathbb{E}^{(k)}(f)}{\ud k}
  =&\!\lim_{\Delta k\to0}\!\frac{\mathbb{E}^{(k+\Delta k)}(f)-\mathbb{E}^{(k)}(f)}{\Delta k} \notag\\
  =&\mathbb{E}^{(k)}\big(f\log(\tau)\big)
  \!-\!\mathbb{E}^{(k)}(f)\mathbb{E}^{(k)}\big(\log(\tau)\big).
\end{align}
Further, let $y=f(x)$, then $\log(\tau(x))=\log(\rho(y))$, and then
\begin{equation*}
  \mathbb{E}^{(k)}\big(f\log(\tau)\big)
  -\mathbb{E}^{(k)}(f)\mathbb{E}^{(k)}\big(\log(\tau)\big)
  =\mathbb{E}\big(y\log(\rho(y))\big)
  -\mathbb{E}(y)\mathbb{E}\big(\log(\rho(y))\big),
\end{equation*}
since $\log(\rho(y))$ is monotonically decreasing, it holds from Lemma \ref{MD:lem:Gurland} that
\begin{equation*}
  \mathbb{E}\big(y\log(\rho(y))\big)\leqslant
  \mathbb{E}(y)\mathbb{E}\big(\log(\rho(y))\big),
\end{equation*}
thus, we have $\frac{\ud\mathbb{E}^{(k)}(f)}{\ud k}\leqslant0$, as claimed. Specially, if $\tau(x)=e^{-f(x)}$, then
\begin{equation*}
  \mathbb{E}^{(k)}\big(f\log(\tau)\big)
  -\mathbb{E}^{(k)}(f)\mathbb{E}^{(k)}\big(\log(\tau)\big)
  =-\mathbb{E}^{(k)}\big(f^2\big)+\big(\mathbb{E}^{(k)}(f)\big)^2,
\end{equation*}
that is,
\begin{equation*}
  \frac{\ud\mathbb{E}^{(k)}(f)}{\ud k}
  =-\int_\Omega\left(f(x)-\mathbb{E}^{(k)}(f)\right)^2m^{(k)}(x)\ud x
  =-\mathbb{V}ar^{(k)}(f),
\end{equation*}
and it is clear that $\mathbb{V}ar^{(k)}(f)>0$ if the continuous function $f$ is not a constant function on $\Omega$, so the proof is complete.\qed
\end{proof}

From the Theorems \ref{MD:thm:NN} and \ref{MD:thm:MC} it follows as an immediate corollary that
\begin{corollary}\label{MD:cor:MC}
For all $k\in\mathbb{R}$ and $\Delta k>0$, it holds that
\begin{equation*}
  \int_\Omega f(x)m^{(k)}(x)\ud x\geqslant
  \int_\Omega f(x)m^{(k+\Delta k)}(x)\ud x\geqslant f^*;
\end{equation*}
further, if $\tau$ is the exponential-type $e^{-f}$ and $f$ is not a constant function on $\Omega$, then
\begin{equation*}
  \int_\Omega f(x)m^{(k)}(x)\ud x>
  \int_\Omega f(x)m^{(k+\Delta k)}(x)\ud x>f^*.
\end{equation*}
\end{corollary}

Further, the following conclusion give a sufficient condition that $\int_\Omega x\cdot m^{(k)}(x)\ud x$ monotonously converge to $x^*$.
\begin{theorem}\label{MD:thm:MCX}
If $x^*$ is the unique global minimizer of $f$ in $\Omega$, $f(x-x^*)=\psi(\|x-x^*\|_2)$ is radial and $\psi$ is a non-decreasing function on $\mathbb{R}_+$, then
\begin{equation*}
  \frac{\ud}{\ud k}\left\|\int_\Omega(x-x^*)
  m^{(k)}(x)\ud x\right\|_2\leqslant0.
\end{equation*}
\end{theorem}
\begin{proof}
Let $r=\|x-x^*\|_2$, then
\begin{equation*}
  \frac{\ud}{\ud k}\left\|\int_\Omega(x-x^*)m^{(k)}(x)\ud x\right\|_2
  \leqslant\frac{\ud}{\ud k}\int_\Omega r\cdot m^{(k)}(x)\ud x,
\end{equation*}
then it holds from \eqref{MD:eq:Ef} that
\begin{equation*}
  \frac{\ud}{\ud k}\int_{\Omega_i}r\cdot m^{(k)}(x)\ud x
  =\mathbb{E}^{(k)}\big[r\log\big(\rho(\psi(r))\big)\big]\!-\!
  \mathbb{E}^{(k)}(r)\mathbb{E}^{(k)}\big[\log\big(\rho(\psi(r))\big)\big],
\end{equation*}
notice that $\log\big(\rho(\psi(r))\big)$ is monotonically decreasing, the desired result follows from Lemma \ref{MD:lem:Gurland}.\qed
\end{proof}

\subsection{Minima distribution}

Now we define a \emph{minima distribution} to be a weak limit $m_{f,\Omega}$ such that the identity
\begin{equation*}
  \int_\Omega m_{f,\Omega}(x)\varphi(x)\ud x=
  \lim_{k\to\infty}\int_\Omega m^{(k)}(x)\varphi(x)\ud x
\end{equation*}
holds for every smooth function $\varphi$ with compact support in $\Omega$. Here are three immediate properties of $m_{f,\Omega}$:
\begin{theorem}
The minima distribution satisfies the following properties:
\begin{enumerate}[(i)]
  \item $m_{f,\Omega}$ satisfies the identity $\int_\Omega m_{f,\Omega}(x)\ud x=1$.
  \item If $f$ is continuous on $\Omega$, then $f^*=\int_\Omega f(x)m_{f,\Omega}(x)\ud x$.
  \item If $x^*$ is the unique global minimizer of $f$ in any $\Omega$, then
      \begin{equation*}
        \int_\Omega x\cdot m_{f,\Omega}(x)\ud x
        =x^*\int_\Omega m_{f,\Omega}(x)\ud x.
      \end{equation*}
\end{enumerate}
\end{theorem}

A naive view of the minima distribution is that $m_{f,\Omega}$ is the pointwise limit
\begin{equation*}
  m_{f,\Omega}(x)=\lim_{k\to\infty}m^{(k)}(x)
\end{equation*}
that also reserves $\int_\Omega m_{f,\Omega}(x)\ud x=1$.
From this view, we begin by noting that $m_{f,\Omega}(x)=0$ if $x\notin X^*$ and deduce that $f(x)m_{f,\Omega}(x)=f^*\cdot m_{f,\Omega}(x)$ for every $x\in\Omega$, then we also have $\int_\Omega f(x)m_{f,\Omega}(x)\ud x =f^*\int_\Omega m_{f,\Omega}(x)\ud x=f^*$, and so on.

\subsection{Stability}
\label{MD:s2:stability}

Suppose $f\in C^2(\Omega)$. According to the first and second order optimality conditions, for any $x_i^*\in X^*$, it holds that
\begin{equation*}
  \nabla f(x_i^*)=0~~\textrm{and}~~\nabla^2f(x_i^*)\succ0,
\end{equation*}
then $f(x)$ is approximately equal to
\begin{equation*}
  f^*+\frac{1}{2}(x-x^*_i)^\mathrm{T}\nabla^2f(x^*_i)(x-x^*_i)
\end{equation*}
in a sufficiently small neighborhood $\Omega_i$ of $x^*_i$. So if $\nabla^2f(x_1^*)\succ\nabla^2f(x_2^*)\succ0$, then for the same small disturbance $\Delta x\in\mathbb{R}^n$, the change of $f(x)$ around $x_1^*$ will be larger than that around $x_2^*$. Usually, we say that $x_2^*$ is more stable than $x_1^*$. In the following, we will see how the minima distribution is associated with the stability.

If $\mu(X^*)\neq0$, then $m_{f,\Omega}$ can be viewed as the PDF of the uniform distribution on $X^*$; if $x^*$ is the unique minimizer of $f$ on $\Omega$, then $m_{f,\Omega}$ is exactly the Dirac delta function $\delta(x-x^*)$; and if $X^*$ is a finite set, i.e., $X^*=\{x^*_i\}_{i=1}^s$, then $m_{f,\Omega}$ can be given by a linear combination of Dirac delta functions
\begin{equation}\label{MD:eq:taom}
  m_{f,\Omega}(x)=\sum_{i=1}^sw_i\delta(x-x^*_i)~~~\textrm{with}~~~
  \sum_{i=1}^sw_i=1~~\textrm{and}~~w_i>0.
\end{equation}
According to (iv) of Theorem \ref{MD:thm:nmd}, if $f\in C^2(\Omega)$ and $\nabla^2f(x^*_1)\succ\nabla^2f(x^*_2)\succ\cdots\succ\nabla^2f(x^*_s)\succ0$, then for given $k\in\mathbb{N}$, it follows that
\begin{equation*}
  \nabla^2m^{(k)}(x^*_1)\prec\nabla^2m^{(k)}(x^*_2)\prec\cdots
  \prec\nabla^2m^{(k)}(x^*_s)\prec0;
\end{equation*}
and further notice that $m^{(k)}_{\Omega,f}$ is approximately equal to
\begin{equation*}
  m^{(k)}(x^*_i)+\frac{1}{2}(x-x^*_i)^\mathrm{T}
  \nabla^2m^{(k)}(x^*_i)(x-x^*_i)
\end{equation*}
in a sufficiently small neighborhood $\Omega_i$ of $x^*_i$.

\begin{figure}[!h]
\centering
\includegraphics[width=0.24\textwidth]{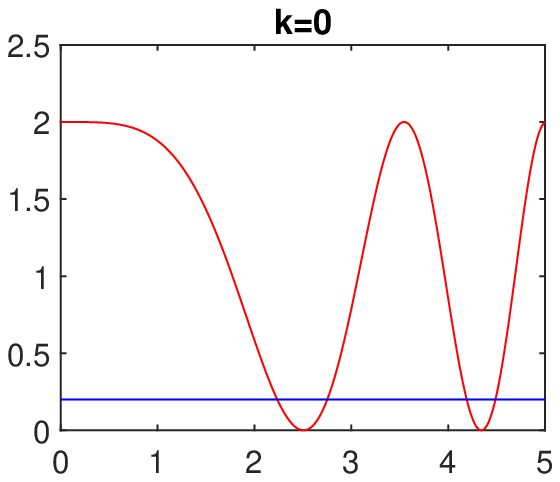}
\includegraphics[width=0.24\textwidth]{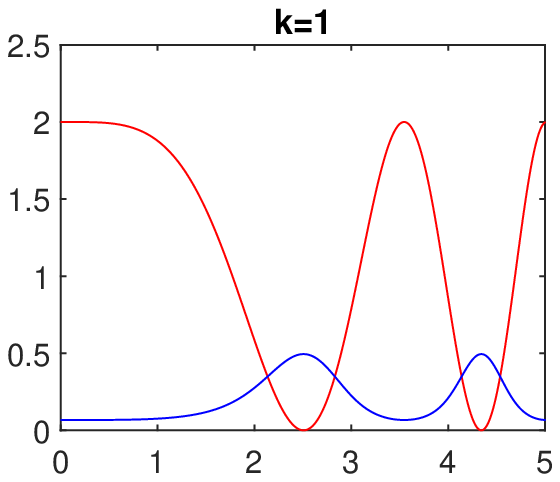}
\includegraphics[width=0.24\textwidth]{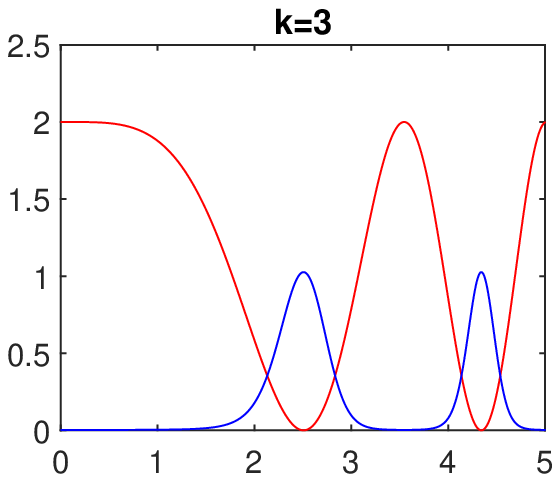}
\includegraphics[width=0.24\textwidth]{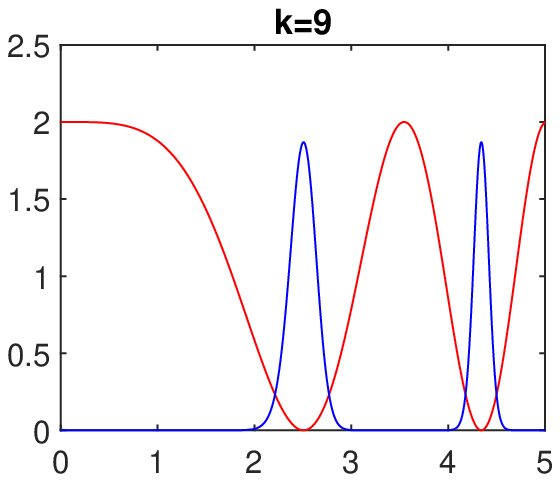}
\caption{Two-dimensional illustration of the stability. The plots show the target function $f=\cos(0.5x^2)+1,x\in[0,5]$ in red and the relevant exponential-type nascent MD functions $m^{(k)}$ with $k=0,1,3,9$ in blue.}
\label{MD:fig:s1}
\end{figure}

\begin{figure}[!h]
\centering
\subfigure[$f$]{\includegraphics[width=0.24\textwidth]{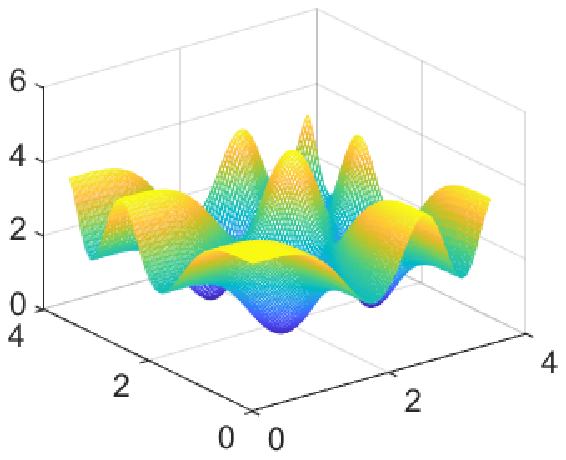}}
\subfigure[$m^{(1)}$]{\includegraphics[width=0.24\textwidth]{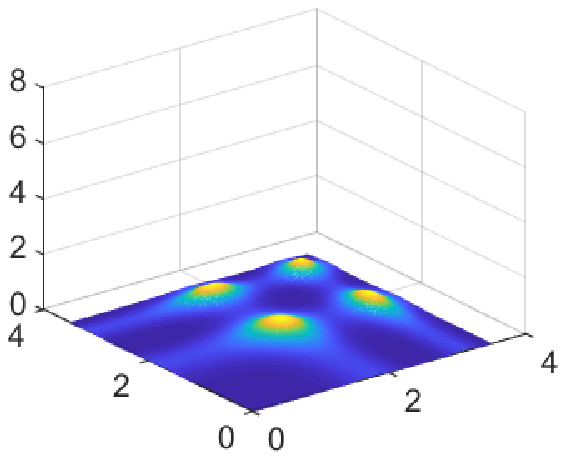}}
\subfigure[$m^{(3)}$]{\includegraphics[width=0.24\textwidth]{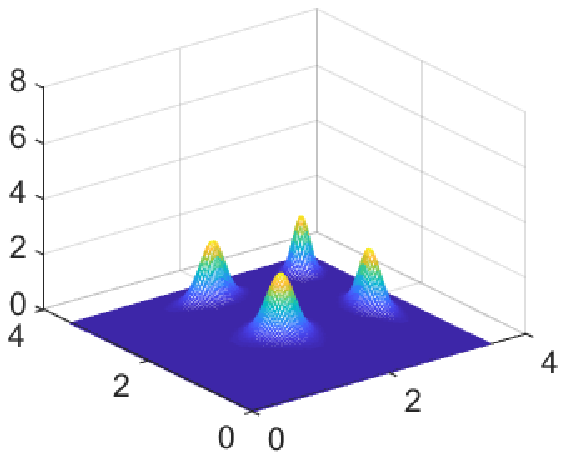}}
\subfigure[$m^{(9)}$]{\includegraphics[width=0.24\textwidth]{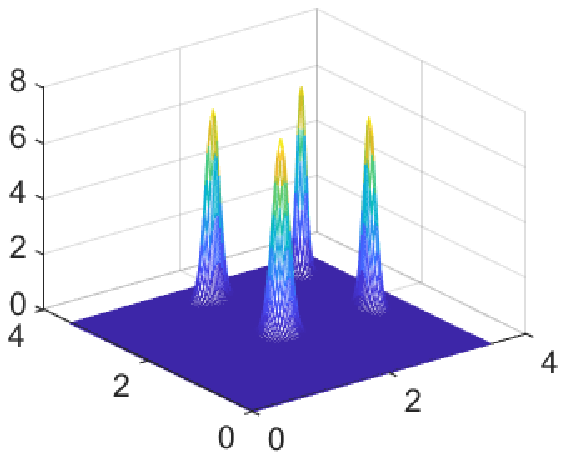}}
\caption{Two-dimensional illustration of the nascent MD functions. The plot (a) shows the target function $f(x)=\cos(x_1^2)+\cos(x_2^2)+2, x=(x_1,x_2)\in[0,3.5]^2$, and the plots (b)-(d) show the relevant exponential-type nascent MD functions $m^{(k)}$ with $k=1,3,9$.}
\label{MD:fig:s2}
\end{figure}

Hence, if all points $\{x_i^*\}_{i=1}^s$ are not on the boundary of $\Omega$, we obtain
\begin{equation*}
  \int_{\Omega_1}m^{(k)}(x)\ud x<\int_{\Omega_2}m^{(k)}(x)\ud x
  <\cdots<\int_{\Omega_s}m^{(k)}(x)\ud x,
\end{equation*}
and further, the weight coefficients of \eqref{MD:eq:taom} satisfy $w_1<w_2<\cdots<w_s$, as illustrated in Figs. \ref{MD:fig:s1} and \ref{MD:fig:s2}. So we have the following theorem:
\begin{theorem}
Suppose $f\in C^2(\Omega)$, $\nabla^2f(x_1^*)\succ\nabla^2f(x_2^*)\succ0$ and two points $x_1^*,x_2^*\in X^*$ are not on the boundary of $\Omega$. If $\xi$ is a random sample from a distribution with PDF $m_{f,\Omega}$, then the probability that $\xi$ takes $x_2^*$ is greater than the probability that $\xi$ takes $x_1^*$.
\end{theorem}

\section{Significant sets}
\label{MD:s3}

\subsection{Significant sets of $m^{(k)}$}
\label{MD:s3:ss}

Here we will introduce several types of significant set of $m^{(k)}$ where each of them monotonically shrinks from the original compact set $\Omega$ to the set of all global minimizers $X^*$. Let's define the first type of significant set of $m^{(k)}$ as
\begin{equation}\label{MD:eq:Df}
  D_f^{(k)}=\left\{x\in\Omega:f(x)\leqslant\mathbb{E}^{(k)}(f)\right\},
\end{equation}
then we have the following containment relationship:
\begin{theorem}[monotonic shrinkage]\label{MD:thm:Df}
Suppose $\Omega$ is a compact set and $f\in C(\Omega)$. For all $k\in\mathbb{R}$ and $\Delta k>0$, it holds that
\begin{equation*}
  \Omega\supseteq D_f^{(k)}\supseteq D_f^{(k+\Delta k)}\supseteq X^*,
\end{equation*}
where $X^*$ is the set of all global minimizers of $f$; especially, if $\tau$ is the exponential-type $e^{-f}$ and $f$ is not a constant on $\Omega$, it holds that $\Omega\supset D_f^{(k)}\supset D_f^{(k+\Delta k)}\supset X^*$.
\end{theorem}
\begin{proof}
It follows from Theorem \ref{MD:thm:NN} that for any $k\in\mathbb{R}$,
\begin{equation*}
  f^*\leqslant\mathbb{E}^{(k)}(f),
\end{equation*}
that is, $X^*\subseteq D_f^{(k)}$ for every $k\in\mathbb{R}$, and it becomes an equality if and only if $f$ is a constant function on $\Omega$. And according to Theorem \ref{MD:thm:MC}, for any $\Delta k>0$, we have
\begin{equation*}
  \mathbb{E}^{(k+\Delta k)}(f)\leqslant\mathbb{E}^{(k)}(f),
\end{equation*}
if $x\in D_f^{(k+\Delta k)}$, then
\begin{equation*}
  f(x)\leqslant\mathbb{E}^{(k+\Delta k)}(f)\leqslant\mathbb{E}^{(k)}(f),
\end{equation*}
that is, $x\in D_f^{(k)}$, thus, $D_f^{(k+\Delta k)}\subseteq D_f^{(k)}$. Moreover, if
\begin{equation*}
  f(x)<\mathbb{E}^{(k+\Delta k)}(f)<\mathbb{E}^{(k)}(f),
\end{equation*}
then $D_f^{(k+\Delta k)}\subset D_f^{(k)}$, where the conditions are consistent with those in Theorem \ref{MD:thm:MC}, and the proof is complete.\qed
\end{proof}

Theorem \ref{MD:thm:Df} can also be stated in another way by replacing the $f$ in the definition \eqref{MD:eq:Df} with $\tau$. First define the second type of significant set of $m^{(k)}$ as
\begin{equation}\label{MD:eq:Dtau}
  D_\tau^{(k)}=\left\{x\in\Omega:
  \tau(x)\geqslant\mathbb{E}^{(k)}(\tau)\right\}.
\end{equation}

It holds from \eqref{MD:eq:Ef} that
\begin{equation*}
  \frac{\ud\mathbb{E}^{(k)}(\tau)}{\ud k}
  =\mathbb{E}^{(k)}\big(\tau\log(\tau)\big)
  -\mathbb{E}^{(k)}(\tau)\mathbb{E}^{(k)}\big(\log(\tau)\big);
\end{equation*}
let $y=\tau(x)$, then we have
\begin{equation*}
  \frac{\ud\mathbb{E}^{(k)}(\tau)}{\ud k} =\mathbb{E}\big(y\log(y)\big) -\mathbb{E}(y)\mathbb{E}\big(\log(y)\big),
\end{equation*}
and it holds from Lemma \ref{MD:lem:Gurland} that
\begin{equation*}
  \mathbb{E}\big(y\log(y)\big) \geqslant\mathbb{E}(y)\mathbb{E}\big(\log(y)\big)
\end{equation*}
thus, $\frac{\ud\mathbb{E}^{(k)}(\tau)}{\ud k}\geqslant0$, that is,
\begin{equation*}
  \mathbb{E}^{(k)}(\tau)\leqslant
  \mathbb{E}^{(k\Delta k)}(\tau)\leqslant\tau*,
\end{equation*}
where $\tau*=\tau(x^*)$ for any $x^*\in X^*$. Hence, we similarly have the containment relationship:
\begin{theorem}[monotonic shrinkage]\label{MD:thm:Dtau}
Suppose $\Omega$ is a compact set and $f\in C(\Omega)$. For all $k\in\mathbb{R}$ and $\Delta k>0$, it holds that
\begin{equation*}
  \Omega\supseteq D_\tau^{(k)}\supseteq
  D_\tau^{(k+\Delta k)}\supseteq X^*,
\end{equation*}
where $X^*$ is the set of all global minimizers of $f$; especially, if $\tau$ is the exponential-type $e^{-f}$ and $f$ is not a constant function on $\Omega$, it holds that $\Omega\supset D_\tau^{(k)}\supset D_\tau^{(k+\Delta k)}\supset X^*$.
\end{theorem}
\begin{remark}\label{MD:rem:eqCond}
It is worth noting that the discriminant condition $\tau(x)\geqslant \mathbb{E}^{(k)}(\tau)$ is equivalent to $m^{(k+1)}(x)\geqslant m^{(k)}(x)$.
\end{remark}

Finally, we will introduce the third type of significant set of $m^{(k)}$ which can be used to estimate the relevant shrink rate. Let
\begin{equation}\label{MD:eq:D0}
  D_0^{(k)}=\left\{x\in\Omega:m^{(k)}(x)\geqslant m^{(0)}(x)=\frac{1}{\mu(\Omega)}\right\}
\end{equation}
with its boundary
\begin{equation}\label{MD:eq:G0}
  \Gamma_0^{(k)}=\left\{x\in\Omega: m^{(k)}(x)=\frac{1}{\mu(\Omega)}\right\}.
\end{equation}
According to Remark \ref{MD:rem:eqCond}, the definition of $D_0^{(k)}$ is actually extended from that of $D_\tau^{(k)}$. And we also have the containment relationship:
\begin{theorem}[monotonic shrinkage]\label{MD:thm:D0}
Suppose $\Omega$ is a compact set and $f\in C(\Omega)$ is not a constant function. For all $k\geqslant0$ and $\Delta k>0$, it holds that
\begin{equation*}
  \Omega=D_0^{(0)}\supset D_0^{(k)}\supset D_0^{(k+\Delta k)}\supset X^*,
\end{equation*}
where $X^*$ is the set of all global minimizers of $f$.
\end{theorem}
\begin{proof}
First, it follows from $m^{(0)}=1/\mu(\Omega)$ that $\Omega=D_0^{(0)}$; and for any $x^*\in X^*$ and $k>0$, it holds from $\int_{\Omega-X^*} \tau^{-k}(x^*)\tau^k(t)\ud t<\mu(\Omega-X^*)$ that
\begin{equation*}
  m^{(k)}(x^*)=\frac{1}{\int_{X^*}\ud t
  +\int_{\Omega-X^*}\tau^{-k}(x^*)\tau^k(t)\ud t}
  >\frac{1}{\mu(X^*)+\mu(\Omega-X^*)}=\frac{1}{\mu(\Omega)},
\end{equation*}
that is, $X^*\subset D_0^{(k)}$ for every $k\geqslant0$.

Then, for any $\Delta k>0$, according to H\"{o}lder's inequality, it follows that
\begin{equation}\label{MD:eq:Holder}
  \int_\Omega\tau^k(t)\ud t<
  \left(\int_\Omega\ud t\right)^{\frac{\Delta k}{k+\Delta k}}
  \left(\int_\Omega\tau^{k+\Delta k}(t)\ud t\right)^{\frac{k}{k+\Delta k}}
\end{equation}
when $f$ is not a constant function on $\Omega$. If $x\in D_0^{(k+\Delta k)}$, then
\begin{equation*}
  \tau^{k+\Delta k}(x)\geqslant
  \frac{\int_\Omega\tau^{k+\Delta k}(t)\ud t}{\int_\Omega\ud t},
\end{equation*}
together with \eqref{MD:eq:Holder}, we have
\begin{equation*}
  \tau^k(x)\!\geqslant\!\left(\frac{\int_\Omega\tau^{k+\Delta k}(t)\ud t}
  {\int_\Omega\ud t}\right)^{\frac{k}{k+\Delta k}}
  \!=\!\frac{\left(\int_\Omega\ud t\right)^{\frac{\Delta k}{k+\Delta k}}
  \!\!\left(\!\int_\Omega\!\tau^{k+\Delta k}(t)\ud t\!
  \right)^{\frac{k}{k+\Delta k}}}{\int_\Omega\ud t}
  \!>\!\frac{\int_\Omega\tau^k(t)\ud t}{\int_\Omega\ud t},
\end{equation*}
that is, $x\in D_0^{(k)}-\Gamma_0^{(k)}$, thus, $D_0^{(k)}\supset D_0^{(k+\Delta k)}$ for any $\Delta k>0$, as claimed.\qed
\end{proof}

\subsection{Shrink rate of $D_0^{(k)}$}
\label{MD:s3:sr}

The shrinkage from $D_0^{(k)}$ to $D_0^{(k+\Delta k)}$ reflects exactly the difference between $m^{(k)}$ and $m^{(k+\Delta k)}$; however, the shrink rate would be slow for a large $k$. Now we will consider the shrink rate.

Suppose that $f\in C^1(\Omega)$, $x\in\Gamma_0^{(k)}$ and $x$ moves to $x'=x+\Delta x\in\Gamma_0^{(k+\Delta k)}$ when $k$ continuously increases to $k'=k+\Delta k$. According to the definition \eqref{MD:eq:G0} of $\Gamma_0^{(k)}$, it is clear that
\begin{equation}\label{MD:eq:CR1}
  m^{(k)}(x)=m^{(k')}(x')=\frac{1}{\mu(\Omega)},
\end{equation}
And it follows from (iii) of Theorem \ref{MD:thm:nmd} that
\begin{align}\label{MD:eq:CR2}
  m^{(k')}(x')=&m^{(k')}(x+\Delta x) \notag\\
  =&m^{(k')}(x)+\nabla m^{(k')}(x)\cdot\Delta x+\bm{o}(\|\Delta x\|)\notag\\
  =&m^{(k')}(x)+\frac{k'm^{(k')}(x)}{\tau(x)}\nabla\tau(x)\cdot\Delta x
  +\bm{o}(\|\Delta x\|),
\end{align}
where $\|\cdot\|$ is the Euclidean norm. Moreover, it follows from (v) of Theorem \ref{MD:thm:nmd} that
\begin{align}\label{MD:eq:CR3}
  m^{(k')}(x)=&m^{(k+\Delta k)}(x) \notag\\
  =&m^{(k)}(x)+\frac{\ud}{\ud k}m^{(k)}(x)\Delta k+\bm{o}(\Delta k)\notag\\
  =&m^{(k)}(x)+m^{(k)}(x)\Big(\log(\tau(x))-
  \mathbb{E}^{(k)}(\log(\tau))\Big)\Delta k+\bm{o}(\Delta k).
\end{align}
Hence, it holds from \eqref{MD:eq:CR1} - \eqref{MD:eq:CR3} that
\begin{equation*}
  1\!=\!\left(\!1\!+\!\frac{k\!+\!\Delta k}{\tau(x)}
  \nabla\tau(x)\!\cdot\!\Delta x\right)
  \!\!\bigg(1+Q(x)\Delta k+\bm{o}(\Delta k)\bigg)+\bm{o}(\|\Delta x\|),
\end{equation*}
where $Q(x)=\log(\tau(x))-\mathbb{E}^{(k)}(\log(\tau))$; that is,
\begin{equation*}
  \frac{k\!+\!\Delta k}{\tau(x)}\Big(1+Q(x)\Delta k\Big)
  \nabla\tau(x)\!\cdot\!\Delta x
  =-Q(x)\Delta k+\bm{o}(\Delta k)+\bm{o}(\|\Delta x\|),
\end{equation*}
then we have
\begin{equation}\label{MD:eq:CR4}
  \frac{\Delta x}{\Delta k}
  =-\frac{1}{k\!+\!\Delta k}\frac{\tau(x)Q(x)}{1+Q(x)\Delta k}
  \frac{\nabla\tau(x)}{\|\nabla\tau(x)\|^2}+\bm{o}(1)
  +\bm{o}\left(\frac{\|\Delta x\|}{\Delta k}\right),
\end{equation}
and further, we can obtain the shrink rate as follows:
\begin{theorem}\label{MD:thm:SR}
Suppose that $f\in C(\Omega)$. If $x\in\Gamma_0^{(k)}$ and $x$ moves to $x+\Delta x\in\Gamma_0^{(k+\Delta k)}$ when $k$ continuously increases to $k+\Delta k$, then the shrink rate
\begin{equation}\label{MD:eq:SR}
  \lim_{\Delta k\to0}\frac{\|\Delta x\|}{\Delta k}
  =\frac{\tau(x)}{k\|\nabla\tau(x)\|}
  \left|\mathbb{E}^{(k)}(\log(\tau))-\log(\tau(x))\right|,
\end{equation}
especially, if $\tau$ is the exponential-type $e^{-f}$, then
\begin{equation*}
  \lim_{\Delta k\to0}\frac{\|\Delta x\|}{\Delta k}
  =\frac{1}{k\|\nabla f(x)\|}\left|\mathbb{E}^{(k)}(f)-f(x)\right|,
\end{equation*}
where $\|\cdot\|$ is the Euclidean norm.
\end{theorem}
\begin{remark}
So the shrink rate of $m^{(k)}$ is inversely proportional to $k$. Since $\int_k^{ek}\frac{1}{s}\ud s=1$, the relevant shrinkage is almost independent of $k$ when $\Delta k=(e-1)k$.
\end{remark}

From this, we can immediately obtain the following rate which can be viewed as a bound of the relevant convergence rate.
\begin{theorem}
Under the assumption of Theorem \ref{MD:thm:SR}, if $x\in\Gamma_0^{(k)}$ and $x$ moves to $x'=x+\Delta x\in\Gamma_0^{(k+\Delta k)}$ when $k$ continuously increases to $k+\Delta k$, then
\begin{equation*}
  \lim_{\Delta k\to0}\frac{f(x)-f(x+\Delta x)}{\Delta k}
  =\frac{\tau(x)}{k}\frac{\nabla f(x)\!\cdot\!\nabla\tau(x)}{\|\nabla\tau(x)\|^2}
  \left(\mathbb{E}^{(k)}(\log(\tau))-\log(\tau(x))\right),
\end{equation*}
especially, if $\tau$ is the exponential-type $e^{-f}$, then
\begin{equation*}
  \lim_{\Delta k\to0}\frac{f(x)-f(x+\Delta x)}{\Delta k}
  =\frac{f(x)-\mathbb{E}^{(k)}(f)}{k}.
\end{equation*}
\end{theorem}
\begin{proof}
By noting that
\begin{equation*}
  \frac{f(x)-f(x+\Delta x)}{\Delta k}
  =-\nabla f(x)\!\cdot\!\frac{\Delta x}{\Delta k}
  +\bm{o}\left(\frac{\|\Delta x\|}{\Delta k}\right),
\end{equation*}
together with \eqref{MD:eq:CR4}, the desired conclusion follows.\qed
\end{proof}

\section{A further remark on the minima distribution}
\label{MD:s4}

As mentioned above, the minima distribution is regarded as a weak limit of a sequence of nascent minima distribution functions. Here we will provide another different way of constructing the nascent minima distribution functions.

Similar to the Dirac delta function, $m^{(k)}$ can also be defined by a uniform distribution sequence. For any $k\in\mathbb{N}_0$, we recursively define the sequence of sets
\begin{equation*}
  D_U^{(0)}=\Omega~~\textrm{and}~~D_U^{(k+1)}=\left\{x\in D_U^{(k)}:
  f(x)\leqslant\frac{1}{\mu(D_U^{(k)})}\int_{D_U^{(k)}}f(t)\ud t\right\},
\end{equation*}
where $\mu\big(D_U^{(k)}\big)=\int_{D_U^{(k)}}\ud t$ is the $n$-dimensional Lebesgue measure of $D_U^{(k)}$; then a uniform distribution based definition can be given as
\begin{equation*}
  m^{(k)}_U(x)=\left\{\begin{array}{cl}
    \frac{1}{\mu(D_U^{(k)})}, & x\in D_U^{(k)};\\
    0, & x\in \Omega-D_U^{(k)}.
  \end{array}\right.
\end{equation*}
And it is clear that
\begin{equation*}
  \frac{1}{\mu(D_U^{(k)})}\int_{D_U^{(k)}}f(t)\ud t=
  \int_\Omega f(t)m^{(k)}_U(t)\ud t.
\end{equation*}

Similarly, we have the global convergence
\begin{equation*}
  \lim_{k\to\infty}\int_\Omega f(t)m^{(k)}_U(t)\ud t=f^*,
  ~~\forall f\in C(\Omega),
\end{equation*}
meanwhile, $\lim_{k\to\infty}m^{(k)}_U(x)$ equals $1/\mu(X^*)$ for every $x\in X^*$ and equals $0$ for every $x\in\Omega-X^*$; and for any $k\in\mathbb{N}_0$, we have the monotonicity
\begin{equation*}
  \int_\Omega f(t)m^{(k)}_U(t)\ud t>\int_\Omega f(t)m^{(k+1)}_U(t)\ud t>f^*
\end{equation*}
if $f$ is not a constant function on $\Omega$, which implies $\Omega=D_U^{(0)}\supset D_U^{(k)}\supset D_U^{(k+1)}\supset X^*$. The construction of $m^{(k)}_U$ also provides an intuitive way to understand the MD theory and the definition of $m^{(k)}$ is not limited to all mentioned in this paper.

\section{Conclusions}
\label{MD:s5}

In this work, we built an MD theory for global minimization of continuous functions on compact sets. In some sense, the proposed theory breaks through the existing gradient-based theoretical framework and allows us to reconsider the non-convex optimization. On the one hand, it can be seen as a way to understand existing algorithms; on the other hand, it may also become a new starting point. Thus, we are convinced that the proposed theory will have a thriving future.


\bibliographystyle{spmpsci} 
\bibliography{MReferences}

\begin{thebibliography}{10}
\providecommand{\url}[1]{{#1}}
\providecommand{\urlprefix}{URL }
\expandafter\ifx\csname urlstyle\endcsname\relax
  \providecommand{\doi}[1]{DOI~\discretionary{}{}{}#1}\else
  \providecommand{\doi}{DOI~\discretionary{}{}{}\begingroup
  \urlstyle{rm}\Url}\fi

\bibitem{AndersonR1953M_RS}
Anderson, R.L.: Recent advances in finding best operating conditions.
\newblock J. Am. Stat. Assoc. \textbf{48}, 789--798 (1953)

\bibitem{BoenderC1982M_Multistart}
Boender, C.G.E., {Rinnooy Kan}, A.H.G., Timmer, G.T., Stougie, L.: A stochastic
  method for global optimization.
\newblock Mathematical Programming \textbf{22}, 125--140 (1982)

\bibitem{BoydS2004B_ConvexOptimization}
Boyd, S., Vandenberghe, L.: Convex Optimization.
\newblock Cambridge University Press, New York (2004)

\bibitem{BrooksS1958M_RS}
Brooks, S.H.: A discussion of random methods for seeking maxima.
\newblock Operations Research \textbf{6}, 244--251 (1958)

\bibitem{BullA2011_BO_EI_Conv}
Bull, A.D.: Convergence rates of efficient global optimization algorithms.
\newblock J. Mach. Learn. Res. \textbf{12}, 2879--2904 (2011)

\bibitem{ByrdR1990M_Multistart}
Byrd, R.H., Dert, C.L., {Rinnooy Kan}, A.H.G., Schnabel, R.B.: Concurrent
  stochastic methods for global optimization.
\newblock Mathematical Programming \textbf{46}, 1--29 (1990)

\bibitem{CoxD1997M_BO_UCB}
Cox, D.D., John, S.: {SDO}: A statistical method for global optimization.
\newblock In: Multidisciplinary Design Optimization: State-of-the-Art, pp.
  315--329 (1997)

\bibitem{DasS2016S_DE}
Das, S., Mullick, S.S., Suganthan, P.N.: Recent advances in differential
  evolution - an updated survey.
\newblock Swarm and Evolutionary Computation \textbf{27}, 1--30 (2016)

\bibitem{DasS2011S_DE}
Das, S., Suganthan, P.N.: Differential evolution: a survey of the
  state-of-the-art.
\newblock IEEE Trans. on Evolutionary Computation \textbf{15(1)}, 4--31 (2011)

\bibitem{FloudasCA2008R_GlobalOptimization}
Floudas, C.A., Gounaris, C.E.: A review of recent advances in global
  optimization.
\newblock Journal of Global Optimization \textbf{45}, 3--38 (2009)

\bibitem{GoldbergD1989_GA}
Goldberg, D.E.: Genetic Algorithms in Search, Optimization, and Machine
  Learning.
\newblock Addison-Wesley, Reading, MA (1989)

\bibitem{GurlandJ1967A_ExpectationInequality}
Gurland, J.: The teacher's corner: an inequality satisfied by the expectation
  of the reciprocal of a random variable.
\newblock The American Statistician \textbf{21(2)}, 24--25 (1967)

\bibitem{JonesD1998M_BO}
Jones, D.R., Schonlau, M., Welch, W.J.: Efficient global optimization of
  expensive black-box functions.
\newblock Journal of Global Optimization \textbf{13}, 455--492 (1998)

\bibitem{KirkpatrickS1983M_SimulatedAnnealing}
Kirkpatrick, S., Gelatt, C.D., Vecchi, M.P.: Optimization by simulated
  annealing.
\newblock Science \textbf{220}, 671--680 (1983)

\bibitem{KushnerH1964_BO_PI}
Kushner, H.J.: A new method of locating the maximum of an arbitrary multipeak
  curve in the presence of noise.
\newblock J. Basic Engineering \textbf{86}, 97--106 (1964)

\bibitem{MasriS1980M_ASR}
Masri, S.F., Bekey, G.A., Safford, F.B.: A global optimization algorithm using
  adaptive random search.
\newblock Applied Mathematics and Computation \textbf{7}, 353--375 (1980)

\bibitem{MitchellM1996_GA}
Mitchell, M.: An Introduction to Genetic Algorithms.
\newblock MIT Press, Cambridge, MA (1996)

\bibitem{MockusJ1974M_BO}
Mockus, J.: On {B}ayesian methods for seeking the extremum.
\newblock Optimization Techniques pp. 400--404 (1974)

\bibitem{MockusJ1978A_BO}
Mockus, J., Tiesis, V., Zilinskas, A.: The application of {B}ayesian methods
  for seeking the extremum.
\newblock Toward Global Optimization \textbf{2}, 117--129 (1978)

\bibitem{MutseniyeksV1964A_RS}
Mutseniyeks, V.A., Rastrigin, L.A.: Extremal control of continuous
  multi-parameter systems by the method of random search.
\newblock Engineering Cybernetics \textbf{1}, 82--90 (1964)

\bibitem{PriceK1996M_DE96}
Price, K., Storn, R.: Minimizing the real functions of the {ICEC}'96 contest by
  differential evolution.
\newblock In: Proceedings of IEEE International Conference on Evolutionary
  Computation (ICEC'96), pp. 842--844 (1996)

\bibitem{PriceK2005_DE}
Price, K., Storn, R., Lampinen, J.: Differential Evolution: A Practical
  Approach to Global Optimization.
\newblock Springer, Berlin, Heidelberg (2005)

\bibitem{RasmussenC2006_GP}
Rasmussen, C.E., Williams, C.: Gaussian Processes for Machine Learning.
\newblock MIT Press, Cambridge, MA (2006)

\bibitem{RastriginL1960M_RS}
Rastrigin, L.A.: Extremal control by the method of random scanning.
\newblock Automation and Remote Control \textbf{21}, 891--896 (1960)

\bibitem{RastriginL1963A_RS}
Rastrigin, L.A.: The convergence of the random search method in the extremal
  control of a many-parameter system.
\newblock Automation and Remote Control \textbf{24}, 1337--1342 (1963)

\bibitem{RechenbergI1973_ES}
Rechenberg, I.: Evolutions strategie: Optimierung technischer Systeme nach
  Prinzipien der biologischen Evolution.
\newblock Frommann-Holzboog, Stuttgart (1973)

\bibitem{RinnooyKan1987M_Multistart1}
{Rinnooy Kan}, A.H.G., Timmer, G.T.: Stochastic global optimization methods
  part i: Clustering methods.
\newblock Mathematical Programming \textbf{39}, 27--56 (1987)

\bibitem{RinnooyKan1987M_Multistart2}
{Rinnooy Kan}, A.H.G., Timmer, G.T.: Stochastic global optimization methods
  part ii: Multi level methods.
\newblock Mathematical Programming \textbf{39}, 57--78 (1987)

\bibitem{RiosLM2013R_GlobalOptimization}
Rios, L.M., Sahinidis, N.V.: Derivative-free optimization: a review of
  algorithms and comparison of software implementations.
\newblock Journal of Global Optimization \textbf{56}, 1247--1293 (2013)

\bibitem{SchrackG1976M_RSstep}
Schrack, G., Choit, M.: Optimized relative step size random searches.
\newblock Mathematical Programming \textbf{10}, 230--244 (1976)

\bibitem{SchumerM1968M_RSstep}
Schumer, M.A., Steiglitz, K.: Adaptive step size random search.
\newblock IEEE Transactions on Automatic Control \textbf{AC-13}, 270--276
  (1968)

\bibitem{SchwefelH1995_ES}
Schwefel, H.P.: Evolution and Optimum Seeking.
\newblock Wiley-Interscience, New York (1995)

\bibitem{ShahriariB2016_BO}
Shahriari, B., Swersky, K., Wang, Z., Adams, R.P., de~Freitas, N.: Taking the
  human out of the loop: A review of {B}ayesian optimization.
\newblock Proceedings of the IEEE \textbf{104}, 148--175 (2016)

\bibitem{SheelaB1979M_RSstep}
Sheela, B.V.: An optimized step size random search ({OSSRS}).
\newblock Computer Methods in Applied Mechanics and Engineering \textbf{19},
  99--106 (1979)

\bibitem{StornR1997M_DE}
Storn, R., Price, K.: Differential evolution - a simple and efficient heuristic
  for global optimization over continuous spaces.
\newblock Journal of Global Optimization \textbf{11}, 341--359 (1997)

\bibitem{YangXS2014B_NIOA}
Yang, X.S.: Nature-Inspired Optimization Algorithms.
\newblock Elsevier Insight, London (2014)

\end{thebibliography}

%
%

\end{document}